\documentclass[12pt]{amsart}
\usepackage{mathrsfs}      
\usepackage{txfonts}
\usepackage{amssymb}
\usepackage{eucal}
\usepackage{graphicx}
\usepackage{amsmath}
\usepackage{amscd}
\usepackage[all]{xy}           
\usepackage[active]{srcltx} 
\usepackage{amsfonts,latexsym}
\usepackage{xspace}
\usepackage{epsfig}
\usepackage{float}
\usepackage{color}
\usepackage{fancybox}
\usepackage{colordvi}
\usepackage{multicol}

\usepackage{pdfsync}

\usepackage{tikz}
\usepackage[hypertex]{hyperref} 

\newtheorem{theorem}{Theorem}[section]
\newtheorem{prop}[theorem]{Proposition}
\newtheorem{defn}[theorem]{Definition}
\newtheorem{lemma}[theorem]{Lemma}
\newtheorem{coro}[theorem]{Corollary}
\newtheorem{prop-def}{Proposition-Definition}[section]
\newtheorem{coro-def}{Corollary-Definition}[section]

\newtheorem{remark}[theorem]{Remark}

\newcommand{\nc}{\newcommand}
\nc{\tred}[1]{\textcolor{red}{#1}}
\nc{\tblue}[1]{\textcolor{blue}{#1}}
\nc{\tgreen}[1]{\textcolor{green}{#1}}
\nc{\tpurple}[1]{\textcolor{purple}{#1}}
\nc{\btred}[1]{\textcolor{red}{\bf #1}}
\nc{\btblue}[1]{\textcolor{blue}{\bf #1}}
\nc{\btgreen}[1]{\textcolor{green}{\bf #1}}
\nc{\btpurple}[1]{\textcolor{purple}{\bf #1}}

\newcommand{\efootnote}[1]{}

\renewcommand{\textbf}[1]{}

\newcommand{\delete}[1]{}

\nc{\mlabel}[1]{\label{#1}}  
\nc{\mcite}[1]{\cite{#1}}  
\nc{\mref}[1]{\ref{#1}}  
\nc{\mbibitem}[1]{\bibitem{#1}} 

\delete{
\nc{\mlabel}[1]{\label{#1}  
{\hfill \hspace{1cm}{\bf{{\ }\hfill(#1)}}}}
\nc{\mcite}[1]{\cite{#1}{{\bf{{\ }(#1)}}}}  
\nc{\mref}[1]{\ref{#1}{{\bf{{\ }(#1)}}}}  
\nc{\mbibitem}[1]{\bibitem[\bf #1]{#1}} 
}

\nc{\opa}{{\circ_1}} \nc{\opb}{{\circ_2}} \nc{\op}{\bullet}
\nc{\rop}{\ast} \nc{\fopa}{\opa}  \nc{\fopb}{\opb}
\nc{\gop}{\circ}
\nc{\pa}{\frakL}
\nc{\arr}{\rightarrow} \nc{\lu}[1]{(#1)} \nc{\mult}{\mrm{mult}}
\nc{\diff}{\mathfrak{Diff}}
\nc{\opc}{\sharp}\nc{\opd}{\natural}
\nc{\ope}{\circ}
\nc{\opf}{\dashv}
\nc{\opg}{\vdash}
\nc{\oph}{\cdot}
\nc{\bin}[2]{ (_{\stackrel{\scs{#1}}{\scs{#2}}})}  
\nc{\binc}[2]{ \left (\!\! \begin{array}{c} \scs{#1}\\
    \scs{#2} \end{array}\!\! \right )}  
\nc{\bincc}[2]{  \left ( {\scs{#1} \atop
    \vspace{-1cm}\scs{#2}} \right )}  
\nc{\bs}{\bar{S}} \nc{\cosum}{\sqsubset} \nc{\la}{\longrightarrow}
\nc{\rar}{\rightarrow} \nc{\dar}{\downarrow} \nc{\dprod}{**}
\nc{\dap}[1]{\downarrow \rlap{$\scriptstyle{#1}$}}
\nc{\md}{\mathrm{dth}} \nc{\uap}[1]{\uparrow
\rlap{$\scriptstyle{#1}$}} \nc{\defeq}{\stackrel{\rm def}{=}}
\nc{\disp}[1]{\displaystyle{#1}} \nc{\dotcup}{\
\displaystyle{\bigcup^\bullet}\ } \nc{\gzeta}{\bar{\zeta}}
\nc{\hcm}{\ \hat{,}\ } \nc{\hts}{\hat{\otimes}}
\nc{\barot}{{\otimes}} \nc{\free}[1]{\bar{#1}}
\nc{\uni}[1]{\tilde{#1}} \nc{\hcirc}{\hat{\circ}} \nc{\lleft}{[}
\nc{\lright}{]} \nc{\lc}{\lfloor} \nc{\rc}{\rfloor}
\nc{\curlyl}{\left \{ \begin{array}{c} {} \\ {} \end{array}
    \right .  \!\!\!\!\!\!\!}
\nc{\curlyr}{ \!\!\!\!\!\!\!
    \left . \begin{array}{c} {} \\ {} \end{array}
    \right \} }
\nc{\longmid}{\left | \begin{array}{c} {} \\ {} \end{array}
    \right . \!\!\!\!\!\!\!}
\nc{\onetree}{\bullet} \nc{\ora}[1]{\stackrel{#1}{\rar}}
\nc{\ola}[1]{\stackrel{#1}{\la}}
\nc{\ot}{\otimes} \nc{\mot}{{{\boxtimes\,}}}
\nc{\otm}{\overline{\boxtimes}} \nc{\sprod}{\bullet}
\nc{\scs}[1]{\scriptstyle{#1}} \nc{\mrm}[1]{{\rm #1}}
\nc{\margin}[1]{\marginpar{\rm #1}}   
\nc{\dirlim}{\displaystyle{\lim_{\longrightarrow}}\,}
\nc{\invlim}{\displaystyle{\lim_{\longleftarrow}}\,}
\nc{\mvp}{\vspace{0.3cm}} \nc{\tk}{^{(k)}} \nc{\tp}{^\prime}
\nc{\ttp}{^{\prime\prime}} \nc{\svp}{\vspace{2cm}}
\nc{\vp}{\vspace{8cm}} \nc{\proofbegin}{\noindent{\bf Proof: }}
\nc{\proofend}{$\blacksquare$ \vspace{0.3cm}}
\nc{\modg}[1]{\!<\!\!{#1}\!\!>}
\nc{\intg}[1]{F_C(#1)} \nc{\lmodg}{\!
<\!\!} \nc{\rmodg}{\!\!>\!}
\nc{\cpi}{\widehat{\Pi}}
\nc{\sha}{{\mbox{\cyr X}}}  
\nc{\shap}{{\mbox{\cyrs X}}} 
\nc{\shpr}{\diamond}    
\nc{\shp}{\ast} \nc{\shplus}{\shpr^+}
\nc{\shprc}{\shpr_c}    
\nc{\msh}{\ast} \nc{\zprod}{m_0} \nc{\oprod}{m_1}
\nc{\vep}{\varepsilon} \nc{\labs}{\mid\!} \nc{\rabs}{\!\mid}

\nc{\mmbox}[1]{\mbox{\ #1\ }} \nc{\fp}{\mrm{FP}}
\nc{\rchar}{\mrm{char}} \nc{\End}{\mrm{End}} \nc{\Fil}{\mrm{Fil}}
\nc{\Mor}{Mor\xspace} \nc{\gmzvs}{gMZV\xspace}
\nc{\gmzv}{gMZV\xspace} \nc{\mzv}{MZV\xspace}
\nc{\mzvs}{MZVs\xspace} \nc{\Hom}{\mrm{Hom}} \nc{\id}{\mrm{id}}
\nc{\im}{\mrm{im}} \nc{\incl}{\mrm{incl}} \nc{\map}{\mrm{Map}}
\nc{\mchar}{\rm char} \nc{\nz}{\rm NZ} \nc{\supp}{\mathrm Supp}
\nc{\GL}{\mathrm{GL}}
\nc{\Alg}{\mathbf{Alg}} \nc{\Bax}{\mathbf{Bax}} \nc{\bff}{\mathbf f}
\nc{\bfk}{{\bf k}} \nc{\bfone}{{\bf 1}} \nc{\bfx}{\mathbf x}
\nc{\bfy}{\mathbf y}
\nc{\base}[1]{\bfone^{\otimes ({#1}+1)}} 
\nc{\Cat}{\mathbf{Cat}}

\nc{\detail}{\marginpar{\bf More detail}
    \noindent{\bf Need more detail!}
    \svp}
\nc{\Int}{\mathbf{Int}} \nc{\Mon}{\mathbf{Mon}}
\nc{\rbtm}{{shuffle }} \nc{\rbto}{{Rota-Baxter }}
\nc{\remarks}{\noindent{\bf Remarks: }} \nc{\Rings}{\mathbf{Rings}}
\nc{\Sets}{\mathbf{Sets}} \nc{\wtot}{\widetilde{\odot}}
\nc{\wast}{\widetilde{\ast}} \nc{\hodot}[1]{\odot^{#1}}
\nc{\hast}[1]{\ast^{#1}} \nc{\mal}{\mathcal{O}}
\nc{\tet}{\tilde{\ast}} \nc{\teot}{\tilde{\odot}}
\nc{\oex}{\overline{x}} \nc{\oey}{\overline{y}}
\nc{\oez}{\overline{z}} \nc{\oea}{\overline{a}}
\nc{\oeb}{\overline{b}} \nc{\oef}{\overline{f}}
\nc{\weast}[1]{\widetilde{\ast}^{#1}}
\nc{\weodot}[1]{\widetilde{\odot}^{#1}} \nc{\hstar}[1]{\star^{#1}}
\nc{\lae}{\langle} \nc{\rae}{\rangle} \nc{\owd}{\overrightarrow{d}}
\nc{\owc}{\overrightarrow{c}} \nc{\mds}{\mathrm{MDS}}
\nc{\mda}{\mathrm{MDA}} \nc{\mdao}{\mathit{MDA}}

\nc{\CC}{\mathbb C}
\nc{\ZZ}{\mathbb{Z}}

\nc{\cala}{{\mathcal A}} \nc{\calb}{{\mathcal B}}
\nc{\calc}{{\mathcal C}}
\nc{\cald}{{\mathcal D}} \nc{\cale}{{\mathcal E}}
\nc{\calf}{{\mathcal F}} \nc{\calg}{{\mathcal G}}
\nc{\calh}{{\mathcal H}} \nc{\cali}{{\mathcal I}}
\nc{\call}{{\mathcal L}} \nc{\calm}{{\mathcal M}}
\nc{\caln}{{\mathcal N}} \nc{\calo}{{\mathcal O}}
\nc{\calp}{{\mathcal P}} \nc{\calr}{{\mathcal R}}
\nc{\cals}{{\mathcal S}} \nc{\calt}{{\mathcal T}}
\nc{\calw}{{\mathcal W}} \nc{\calk}{{\mathcal K}}
\nc{\calx}{{\mathcal X}} \nc{\CA}{\mathcal{A}}

\nc{\fraka}{{\mathfrak a}} \nc{\frakA}{{\mathfrak A}}
\nc{\frakb}{{\mathfrak b}} \nc{\frakB}{{\mathfrak B}}
\nc{\frakD}{{\mathfrak D}} \nc{\frakg}{{\mathfrak g}}
\nc{\frakH}{{\mathfrak H}} \nc{\frakL}{{\mathfrak L}}
\nc{\frakM}{{\mathfrak M}} \nc{\bfrakM}{\overline{\frakM}}
\nc{\frakm}{{\mathfrak m}} \nc{\frakP}{{\mathfrak P}}
\nc{\frakN}{{\mathfrak N}} \nc{\frakp}{{\mathfrak p}}
\nc{\frakS}{{\mathfrak S}}

\font\cyr=wncyr10 \font\cyrs=wncyr7
\nc{\li}[1]{\textcolor{red}{\tt LG:#1}}
\nc{\yong}[1]{\textcolor{blue}{YZ: #1}}
\nc{\cm}[1]{\textcolor{red}{CB: #1}}

\begin{document}

\title{The category and operad of matching dialgebras}
%

\author{Yong Zhang}
\address{Department of Mathematics, Zhejiang University, Hangzhou, Zhejiang 310027, China}
\email{tangmeng@zju.edu.cn}
\author{Chengming Bai}
\address{Chern Institute of Mathematics and LPMC, Nankai University, Tianjin 300071, China}
         \email{baicm@nankai.edu.cn}
\author{Li Guo}\thanks{Corresponding author: Li Guo; e-mail:liguo@rutgers.edu; phone: (973) 353-3917; fax: (973) 353-5270.}
\address{School of Mathematics and Statistics, Lanzhou University, Lanzhou, Gansu 730000, China and Department of Mathematics and Computer Science,
         Rutgers University,
         Newark, NJ 07102, USA}
\email{liguo@rutgers.edu}

\subjclass[2000]{
16Y99,  
18D50,  
18G60,   
17B99.  
}

\keywords{matching dialgebra, matched pairs, semi-direct sum, pre-Lie algebra, PostLie algebra, rewriting method, Koszul operad, free algebra, Quillen homology}



\begin{abstract}
This paper gives a systematic study of matching dialgebras corresponding to the operad $As^{(2)}$ in~\cite{Zi} as the only Koszul self dual operad there other than the operads of associative algebras and Poisson algebras. The close relationship of matching dialgebras with semi-homomorphisms and matched pairs of associative algebras are established. By anti-symmetrizing, matching dialgerbas are also shown to give compatible Lie algebras, pre-Lie algebras and PostLie algebras. By the rewriting method, the operad of matching dialgebras is shown to be Koszul and the free objects are constructed in terms of tensor algebras. The operadic complex computing the homology of the matching dialgebras is made explicit.

\end{abstract}

\maketitle

\tableofcontents

\setcounter{section}{0}

\section{Introduction}
The Encyclopedia of Types of Algebras~\mcite{Zi} is a collection of algebraic structures (operads) together with their known properties. The operad $As^{(2)}$ is one of the few in the collection that have not been studied systematically. This operad in fact has very good properties. For instance, as well as being Koszul, it is also Koszul self dual, the only such structure in~\mcite{Zi} other than the associative algebra and the Poisson algebra. Our motivation of studying this structure, under the name matching dialgebra, came from its close relationship with the concept of a matched pair~\mcite{Bai} of two associative algebras, the latter being a natural generalization of semi-direct sums of an associative algebra with its bimodules.
In fact, the notion of the matched pair of Lie algebras played an important role in the study of Lie bialgebras~\mcite{Maj} and geometric structures on Lie groups, such as complex product structures~\mcite{AS}. Later, a matched pair of associative algebras played an essential role in the study of double constructions of Frobenius algebras and Connes cocycles, which are equivalent to antisymmetric infinitesimal bialgebras and dendriform D-bialgebras respectively~\mcite{Bai}.

In this paper we show that the matching dialgebra has many other interesting properties. We also find further relations between matching dialgebras with semi-homomorphisms, the above mentioned matched pairs and some important Lie type algebras: compatible Lie algebras in integrable systems and Yang-Baxter equations~\mcite{GS1,GS2,Ku,OS3}, pre-Lie algebras in deformations of associative algebras, affine geometry and quantum field theory~\mcite{Bu,Ge,Vi} and PostLie algebras in operads, integrable systems and classical Yang-Baxter equations~\mcite{BGN,Va1}. We construct the graded-module of the operad of matching dialgebras and the free objects from tensor algebras. We show that its operad is Koszul and construct its homology.

In \S~\mref{ss:def}, we give the definition of a matching dialgebra and study its relationship with semi-homomorphisms. The relationship between matching dialgebras and matched pairs of associative algebras is studied in \S~\mref{ss:matp}. In \S~\mref{sec:matlie}, we derive from matching dialgebras some useful Lie type algebras, such as compatible Lie algebras, pre-Lie algebras and PostLie algebras, in a way similar to the process of deriving Lie algebras from associative algebras by anti-symmetrizing.
We apply the rewriting method to prove that the operad $As^{(2)}$ of matching dialgebras is Koszul in \S~\mref{ss:matkos}, and construct the graded space of the operad $As^{(2)}$ and free matching dialgebras in \S~\mref{sec:free}.
The operadic homology of $As^{(2)}$-algebras is constructed in \S~\mref{ss:mathom}.

\section{Matching dialgebras, matched pairs and Lie type algebras}\mlabel{sec:mat2}

Throughout this paper, algebras and modules are assumed to be over a field $\bfk$. The tensor product $\otimes_{{\bf k}}$ over ${\bf k}$ is simply denoted by $\otimes$.

\subsection{Definitions of matching dialgebras}
\mlabel{ss:def}
The following algebra structure is denoted by $As^{(2)}$-algebra in \mcite{Zi} and got its name from its close relationship with matched pairs of associative algebras. See Section~\mref{ss:matp}.
\begin{defn}
{\rm
\begin{enumerate}
\item
A {\bf matching dialgebra (MDA) }is a $\bfk$-module $A$
 with two binary operations:
$$\opa,\opb:\quad A\otimes A\rightarrow A,$$
satisfying the following {\bf MDA axioms}:
\begin{enumerate}
\item
$\opa$ and $\opb$ are associative; 
\item
the following equations hold
\begin{equation}
(x\opa y)\opb z=x\opa(y\opb z),\quad (x\opb y)\opa z=x\opb(y\opa z),\quad
\forall x,y,z \in A.
\mlabel{eq:mda2}
\end{equation}
\end{enumerate}
\mlabel{it:mda1}
\item
Let $(R,\ope_1,\ope_2)$ and $(R',\ope_1',\ope_2')$ be two
matching dialgebras. A linear map $f:R\rightarrow R'$ is a {\bf
homomorphism of matching dialgebras} if $f$ is a {\bfk}-module homomorphism and for all $a$, $b\in R$:
\begin{equation}
f(a\ope_1b)=f(a)\ope_1'f(b)\quad\text{and}\quad
f(a\ope_2b)=f(a)\ope_2'f(b).
\end{equation}
\end{enumerate}
}
\mlabel{def:mda}
\end{defn}

We give some examples of matching dialgebras from semi-homomorphisms.

\begin{defn}{\rm Let $A$ be a $\bfk$-algebra.
\begin{enumerate}
\item
A linear map $f:A\rightarrow A$
is called a {\bf left semi-homomorphism} (resp. {\bf right semi-homomorphism}) if $f$ satisfies
$$
f(xy)=xf(y)\ (\text{resp. } f(xy)=f(x)y), \quad
\forall x,y\in A.
$$
\item
A linear map $f:A\rightarrow A$
is called a {\bf semi-homomorphism} if $f$ is both a left and a right semi-homomorphism.
\end{enumerate}}
\end{defn}

\begin{remark}
{\rm
\begin{enumerate}
\item Let $A=A_1\oplus A_2$ be a direct sum of ideals. Then the projections $p_1$ and $p_2$ of $A$ to $A_1$ and $A_2$
are semi-homomorphisms. In this sense, semi-homomorphisms are
natural generalizations of projections.
\item A left (resp. right) semi-homomorphism is a left (resp. right) Baxter operator, in the sense that $f$ satisfies
$$f(x)f(y)=f(f(x)y)\ (\text{resp. } f(x)f(y)=f(xf(y))),\quad \forall x,y\in A.$$
\end{enumerate}
}
\end{remark}

\begin{prop}
Let $(A,\cdot)$ be a $\bfk$-algebra.
\begin{enumerate}
\item Let $f$ be a left (resp. right) semi-homomorphism. Then the triple $(A,\opa,\opb)$ with
$$
x\opa y:=x\cdot y,\;\;x\opb y=f(x)\cdot y,\;\;(\text{resp.} x\opb y = x\cdot f(y)),\;\;\forall x,y\in A,
$$
is a matching dialgebra.
\mlabel{it:con1}
\item Let $f,g$ be two semi-homomorphisms. Then the triple $(A,\opa,\opb)$ with
\begin{equation}
x\opa y:=f(x)\cdot y(=f(x\cdot y)=x\cdot f(y)),\;\;
x\opb y:=g(x)\cdot y(=g(x\cdot y)=x\cdot g(y)),\;\;\forall\, x,y\in A,
\mlabel{eq:con}
\end{equation}
is a matching dialgebra.
\mlabel{it:con3}
\end{enumerate}
\mlabel{thm:con}
\end{prop}
The proof is a direct computation from the definitions.

\subsection{Matching dialgebras and matched pairs of algebras}
\mlabel{ss:matp}
The notion of a matched pair~\mcite{Bai} of associative algebras is obtained by generalizing a semi-direct sum of an associative algebra and its bimodule to a direct sum of the underlying vector spaces of two associative algebras. It provides a natural underlying structure for the so-called  ``double" structures and hence bialgebra theories. In fact, many bialgebraic structures can be interpreted in terms of matched pairs of the corresponding algebras.

\begin{defn}$($\mcite{Bai}$)$
{\rm Let $(A,\cdot)$ and $(B,\ope)$ be two associative algebras.
 Suppose that there are linear maps $l_A,\,r_A:A\arr gl(B)$
and $l_B,r_B:B\arr gl(A)$ such that $(B,l_A,r_A)$ is a bimodule
of $A$ and $(A,l_B,r_B)$ is a bimodule of $B$. If the following equations are also satisfied, then $(A,B,(l_A,r_A),(l_B,r_B))$
  is called a {\bf matched pair} of associative algebras:
\begin{equation}
l_A(x)(a\ope b)=l_A(r_B(a)x)b+(l_A(x)a)\ope b,
\end{equation}
\begin{equation}
r_A(x)(a\ope b)=r_A(l_B(b)x)a+a\ope(r_A(x)b),
\end{equation}
\begin{equation}
l_B(a)(x\cdot y)=l_B(r_A(x)a)y+(l_B(a)x)\cdot y,
\end{equation}
\begin{equation}
r_B(a)(x\cdot y)=r_B(l_A(y)a)x+x\cdot(r_B(a)y),
\end{equation}
\begin{equation}
l_A(l_B(a)x)b+(r_A(x)a)\ope b-r_A(r_B(b)x)a-a\ope(l_A(x)b)=0,
\end{equation}
\begin{equation}
l_B(l_A(x)a)y+(r_B(a)x)\cdot y-r_B(r_A(y)a)x-x\cdot(l_B(a)y)=0,
\end{equation}
for any $x,y\in\,A$, $a,b\in\,B$.
}
\end{defn}

The close relationship between matching dialgebras and matched pairs of associative algebras is given by the following theorem.

\begin{theorem}
Let $A_{\opa}:=(A,\opa)$ and $A_{\opb}:=(A,\opb)$ be two {\bf k}-algebras. Let $L_{\gop_i}$ (resp. $R_{\gop_i}$) denote the left (resp. right) regular representation of $A_{\gop_i}$, $i=1,2$. Then the following conditions are equivalent:
\begin{enumerate}
\item $(A,\opa,\opb)$ is a matching dialgebra.
\item $(A_{\opa},A_{\opb},(L_{\opa},0),(L_{\opb},0))$ is a
matched pair of associative algebras.
\item $(A_{\opa},A_{\opb},(0, R_{\opa}),(0, R_{\opb}))$ is a
matched pair of associative algebras.
\end{enumerate}
\mlabel{thm:mp}
\end{theorem}

\begin{proof} (a) $\Longleftrightarrow$ (b).
It is easily checked that $(A_{\opa},A_{\opb},(L_{\opa},0),(L_{\opb},0))$ is
a matched pair of associative algebras if and only if the following equations hold:
\begin{eqnarray}
L_{\opa}(x) (a\opb b)&=&(L_{\opa}(x)a)\opb b,
\mlabel{eq:mpa}
\\
L_{\opb}(a)(x\opa y)&=&(L_{\opb}(a)x)\opa y,
\mlabel{eq:mpb}
\\
L_{\opa}(L_{\opb}(a)x)b&=&a\opb(L_{\opa}(x)b),
\mlabel{eq:mpc}
\\
L_{\opb}(L_{\opa}(x)a)y&=&x\opa(L_{\opb}(a)y),\;\; \forall x,y, a,b\in A.
\mlabel{eq:mpd}
\end{eqnarray}
From Eq.~(\mref{eq:mpa}), we have
$$(x\opa a)\opb b=x\opa(a\opb b).$$
From Eq.~(\mref{eq:mpb}), we have
$$(a\opb x)\opa y=a\opb(x\opa y).$$
Thus, if $(A_{\opa},A_{\opb},(L_{\opa},0),(L_{\opb},0))$ is
a matched pair of associative algebras, then $(A,\opa,\opb)$ is a matching dialgebra. Note that in this case,
Eq.~(\mref{eq:mpc}) is equivalent to Eq.~(\mref{eq:mpb}) and
Eq.~(\mref{eq:mpd}) is equivalent to Eq.~(\mref{eq:mpa}).
Conversely, if $(A,\opa,\opb)$ is a matching dialgebra, then Eq.~(\mref{eq:mpa})--Eq.~(\mref{eq:mpd}) hold. Therefore, $(A_{\opa},A_{\opb},(L_{\opb},0),(L_{\opa},0))$ is
a matched pair of associative algebras.
\smallskip

(a) $\Longleftrightarrow$ (c) follows from a similar argument.
\end{proof}

\begin{prop}(\mcite{Bai})\quad
If $(A,B,(l_A,r_A),(l_B,r_B))$ is a matched pair of associative algebras,
then there
is an associative algebra structure on the direct sum $A\oplus B$ of
the underlying vector spaces of $A$ and $B$. The product is given by
\begin{equation}(x,a)*(y,b)=(x\cdot y+l_B(a)y+r_B(b)x,a\circ b+l_A(x)b+r_A(y)a),\;\;\forall x,y\in A,a,b\in B.\end{equation}
Conversely, every associative algebra whose underlying vector space has a decomposition into the
direct sum of two subalgebras can be
obtained from a matched pair of the two subalgebras.
\mlabel{pp:bai2}
\end{prop}

From Theorem~\mref{thm:mp} and Proposition~\mref{pp:bai2} we have the following direct consequence.
\begin{coro}
Let $A$ be a vector space with two binary operations $\opa,\opb:A\ot A \rightarrow A$. Define two binary operations $\cdot,\circ: (A\oplus A)\ot (A\oplus A)\rightarrow (A\oplus A)$ by
\begin{equation} (a,b)\cdot (c,d)=(a\opa c+a\opb d, b\opb d+b\opa c),\;\;\forall a,b,c,d\in A,\mlabel{eq:sum1}\end{equation}
\begin{equation} (a,b)\circ (c,d)=(a\opa c+b\opb c, b\opb d+a\opa d),\;\;\forall a,b,c,d\in A.\mlabel{eq:sum2}\end{equation}
respectively. Then the following conditions are equivalent:
\begin{enumerate}
\item $(A,\opa,\opb)$ is a matching dialgebra.
\item $(A\oplus A,\cdot)$ is an associative algebra.
\item $(A\oplus A,\circ)$ is an associative algebra.
\end{enumerate}
\end{coro}

\subsection{Matching dialgebras and Lie type algebras}
\mlabel{sec:matlie}
We now consider the close relationship between the matching dialgebra and several Lie type algebras that have played important roles in mathematics and mathematical physics as indicated in the introduction. We mainly consider the following binary operations.

\begin{defn}
{\rm Let $(A,\opa,\opb)$ be a matching dialgebra. Define
\begin{equation}
[x,y]_\opa:=x\opa y-y\opa x,\; [x,y]_\opb:=x\opb y-y\opb x,\mlabel{eq:lie}\end{equation}}
\begin{equation}
[x,y]_{\opa,\opb}:=x\opa y-y\opb x,\; [x,y]_{\opb,\opa}:= x\opb
y-y\opa x.\mlabel{eq:plie1}\end{equation}
\end{defn}

\begin{remark} {\rm \begin{enumerate}
\item Since both $(A,\opa)$ and $(A,\opb)$ are associative algebras, $(A, [\;,\;]_\opa)$ and
$(A, [\;,\;]_\opb)$ are Lie algebras.
\item It is obvious that $(A,-[\;,\;]_{\opb,\opa})$ is the opposite algebra of $(A,[\;,\;]_{\opa,\opb})$.
\end{enumerate}}
\mlabel{rk:mdl}
\end{remark}

\begin{defn}
{\rm
\begin{enumerate}
\item (\mcite{Ku,OS1,OS2,OS4})
Two associative algebras $(A,\opa)$ and $(A,\opb)$ are called {\bf compatible} if,
for any $\alpha,\beta\in \bf k$, the product
\begin{equation}
x\star y=\alpha x\opa y+\beta x\opb y, \forall x,y\in A,
\end{equation}
defines an associative algebra.
\item
(\mcite{GS1,GS2,Ku,OS3})
Two Lie algebras $(V,[\ ,\ ]_1)$ and $(V,[\ ,\ ]_2)$ are called {\bf compatible} if,
for any $\alpha,\beta\in \bf k$, the product
\begin{equation}
[x, y]=\alpha [x, y]_1+\beta [x, y]_2,\forall x,y\in V,
\end{equation}
defines a Lie algebra.
\end{enumerate}
}
\end{defn}

\begin{prop}
\begin{enumerate}
\item
Let $(A,\opa,\opb)$ be a matching dialgebra. Then $(A,\opa)$ and $(A,\opb)$ are compatible associative algebras.
\mlabel{it:ascom}
\item
Let $(A,\opa,\opb)$ be a matching dialgebra. Then $(A,[\ ,\ ]_\opa)$ and $(A,[\ , \ ]_\opb)$ defined by Eq.~(\mref{eq:lie}) are compatible Lie algebras.
\mlabel{it:liecom}
\end{enumerate}
\mlabel{pp:com}
\end{prop}

The proof follows from a direct computation.

We next study the relationship between matching dialgebras and pre-Lie algebras~\mcite{Bu,Ge,Vi}.
\begin{lemma}
Let $(A,\opa,\opb)$ be a matching dialgebra. With the notations in Eqs.~(\ref{eq:lie}) and (\ref{eq:plie1}), the following equations hold
\begin{eqnarray}
&&[[x,y]_{\opa,\opb},z]_{\opa,\opb}-[x,[y,z]_{\opa,\opb}]_{\opa,\opb}-([[y,x]_{\opa,\opb},z]_{\opa,\opb}
-[y,[x,z]_{\opa,\opb}]_{\opa,\opb})\notag \\
&=& [x,y]_\opb\opa z-z\opb
[x,y]_\opa, \mlabel{eq:plie2}\\
&&
[[x,y]_{\opb,\opa},z]_{\opb,\opa}-[x,[y,z]_{\opb,\opa}]_{\opb,\opa}-([[y,x]_{\opb,\opa},z]_{\opb,\opa}
-[y,[x,z]_{\opb,\opa}]_{\opb,\opa}) \notag \\
&=& [x,y]_\opa\opb z-z\opa
[x,y]_\opb,\forall x,y,z\in A.\mlabel{eq:plie3}
\end{eqnarray}
\mlabel{lemma:plie2}
\end{lemma}

\begin{proof} For $x,y,z\in A$, we have
\begin{eqnarray*}
&&[[x,y]_{\opa,\opb},z]_{\opa,\opb}-[x,[y,z]_{\opa,\opb}]_{\opa,\opb}=x\opa(z\opb
y)
+(y\opa z)\opb x-z\opb(x\opa y)-(y\opb x)\opa z\,,\\
&&[[y,x]_{\opa,\opb},z]_{\opa,\opb}
-[y,[x,z]_{\opa,\opb}]_{\opa,\opb}=(x\opa z)\opb y+y\opa(z\opb
x)-z\opb(y\opa x)-(x\opb y)\opa z\,.
\end{eqnarray*}
This proves Eq.~(\mref{eq:plie2}). Since $\opa$ and $\opb$ are symmetric in the definition of matching dialgebras, Eq.~(\mref{eq:plie3}) also follows.
\end{proof}

\begin{defn}
{\rm (~\mcite{Bu,Ge,Vi}) A {\bf pre-Lie algebra} is a pair $(A,\ope)$ where $A$ is a {\bf k}-vector space and $\ope$ is a binary operation on $A$ that satisfies
\begin{equation} (x\ope y)\ope z-x\ope(y\ope z)= (y\ope x)\ope z-y\ope(x\ope z),
\forall\,x,y,z\in A. \mlabel{eq:prelie}\end{equation}
If in addition, $(A,\ope)$ satisfies
\begin{equation} (x\ope y)\ope z-x\ope(y\ope z)= (x\ope z)\ope y-x\ope(z\ope y),
\forall\,x,y,z\in A,\mlabel{eq:right-symm}\end{equation}
then $(A,\ope)$ is called a {\bf bi-symmetric algebra} or an {\bf assosymmetric algebra}.}
\end{defn}

\begin{prop}
Let $(A,\opa,\opb)$ be a matching dialgebra with both $\opa$ and $\opb$ commutative. With the notations in Eq.~$($\ref{eq:plie1}$)$,
$(A,[\;,\;]_{\opa,\opb})$ is an assosymmetric algebra.
\end{prop}

\begin{proof} By Lemma~{\mref{lemma:plie2}}, $(A,[\;,\;]_{\opa,\opb})$ and $(A,[\;,\;]_{\opb,\opa})$ are pre-Lie algebras. By \mcite{Bai1}, a pre-Lie algebra is an assosymmetric algebra if and only if
its opposite algebra is also a pre-Lie algebra. Since
$(A,-[\;,\;]_{\opb,\opa})$ is the opposite algebra of
$(A,[\;,\;]_{\opa,\opb})$ by Remark~\mref{rk:mdl}, the conclusion follows.
\end{proof}

Next we consider the case when $\opa$ and $\opb$ might not be commutative.

\begin{defn}{\rm  ~(\mcite{Va1})
A {\bf left PostLie algebra} is a triple $(A,\circ,[\;,\;])$ consisting of a vector space $A$, a product $\circ$
and a skew-symmetric product $[\;,\;]$ satisfying the relations:
\begin{eqnarray}
&[[x,y],z]+[[z,x],y]+[[y,z],x]=0,\mlabel{eq:postlie1}& \\
&(x\circ y)\circ z-x\circ(y\circ z)-(y\circ x)\circ z+y\circ(x\circ z)-[x,y]\circ z =0, \mlabel{eq:postlie2}\\
& z\circ [x,y] -[z\circ x,y]-[x,z\circ y]=0.\mlabel{eq:postlie3}&
\end{eqnarray}
}
\end{defn}

\begin{prop}
Let $(A,\opa,\opb)$ be a matching dialgebra such that
$[x,y]_\opa=[x,y]_\opb$. Then
$(A,[\;,\;]_{\opa,\opb}, [\;,\;]_\opa)$ is a left PostLie algebra.
\mlabel{prop:plie}
\end{prop}

\begin{proof} Eq.~(\mref{eq:postlie1}) follows from the fact that $(A,[\;,\;]_\opa)$ is a Lie algebra. Since $[x,y]_\opa=[x,y]_\opb$, by Eq.~(\ref{eq:plie2}) we have
\begin{eqnarray*}
&&[[x,y]_{\opa,\opb},z]_{\opa,\opb}-[x,[y,z]_{\opa,\opb}]_{\opa,\opb}-([[y,x]_{\opa,\opb},z]_{\opa,\opb}
-[y,[x,z]_{\opa,\opb}]_{\opa,\opb})\\
&=&[[x,y]_\opa,z]_{\opa,\opb}, \quad \forall x,y,z\in A.
\end{eqnarray*}
Thus Eq~(\mref{eq:postlie2}) holds.

Further by definition,
\begin{eqnarray*}
&&
[z,[x,y]_\opa]_{\opa,\opb}-[[z,x]_{\opa,\opb},y]_\opa-[x,[z,y]_{\opa,\opb}]_\opa\\
&=&y\opa (z\opa x)-y\opb (z\opa x)-x\opa (z\opa y)+x\opb (z\opa y)\\
&=&y\opa (z\opa x)-\big((z\opa x)\opb y+y\opa(z\opa x)-(z\opa x)\opa y\big)\\
&&-x\opa (z\opa y)+\big((z\opa y)\opb x+
x\opa (z\opa y)-(z\opa y)\opa x\big)\\
&=&-z\opa (x\opb y)+z\opa (x\opa y)-z\opa (y\opa x)+z\opa (y\opb x)\\
&=&z\opa (-x\opb y+x\opa y-y\opa x+y\opb x)=0.
\end{eqnarray*}
Therefore Eq.~(\mref{eq:postlie3}) holds.
\end{proof}

\section{Koszulity, free objects and operadic homology of the matching dialgebras}
\mlabel{sec:koshom}

We apply the rewriting method to prove that the operad of matching dialgebras is Koszul in Section~\mref{ss:matkos} and to construct free matching dialgebras in Section~\mref{sec:free}.
 We finally give the operadic homology of the matching dialgebras in Section~\mref{ss:mathom}.

\subsection{The operad $As^{(2)}$ of matching dialgebras is Koszul}
\mlabel{ss:matkos}

We recall the following general result~\cite{Ho,LV} on nonsymmetric (ns) operads.

\begin{theorem}
$($Rewriting method for ns operads~\mcite{Ho,LV}$)$. Let $\calp:=\calp(E;R)$ be a reduced quadratic ns operad. If its generating space $E$ admits an ordered basis for which
there exists a suitable order on planar trees such that every critical monomial is
confluent, then the ns operad $\calp$ is Koszul.
\mlabel{thm:lv}
\end{theorem}

{\bf Change of notations:}
To avoid confusion with the conventional notations in an operad, {\it in the rest of the paper}, we will use the notations $\mu_1:=\opa$ and $\mu_2:=\opb$, and reserve $\circ_1:=((\cdot,\cdot),\cdot)$ and $\circ_2:=(\cdot,(\cdot,\cdot))$ for the compositions in the operad $As^{(2)}$.
Thus $f\circ_i g$ is the composition of $g$ on the $i$-th input of $f$, $i=1,2$.

\begin{theorem}
The operad $As^{(2)}$ of matching dialgebras is a Koszul operad.\mlabel{thm:kos}
\end{theorem}

\begin{proof}
The generating space $E$ of binary operations for the operad $As^{(2)}$ has an ordered basis $\{\mu_1>\mu_2\}$ corresponding to the two binary operations in the matching dialgebra.
\smallskip

With the notations just fixed above, the space $R$ of relations of the operad $As^{(2)}$ is spanned by the following set of relators coming from
the matching dialgebra in Definition~\mref{def:mda}:
\begin{equation}
\mu_i\circ_1\mu_j-\mu_i\circ_2\mu_j=0,\quad i, j\in \{1,2\}, \mlabel{eq:re1}
\end{equation}
with the first terms as the leading terms.
These relations define the rewriting rule
$$
\mu_i\circ_1\mu_j \mapsto \mu_i\circ_2\mu_j, , \quad i,j \in\{1,2\},$$
which gives rise to the critical monomials
\begin{equation}
\mu_i\circ_1\mu_j\circ_1\mu_k, \quad i, j, k\in \{1,2\}.
\mlabel{eq:cri}
\end{equation}

By Theorem~\mref{thm:lv}, it is enough to check that these critical monomials are confluent. Since the relations in Eq.~(\mref{eq:re1}) are of associative type, the confluence of the critical monomials can be checked by their diamonds in the pentagon shown in Figure~\ref{fig1}
 in page~\pageref{fig1}.
\begin{figure}[h] \centering

  \includegraphics*[141,275][431,582]{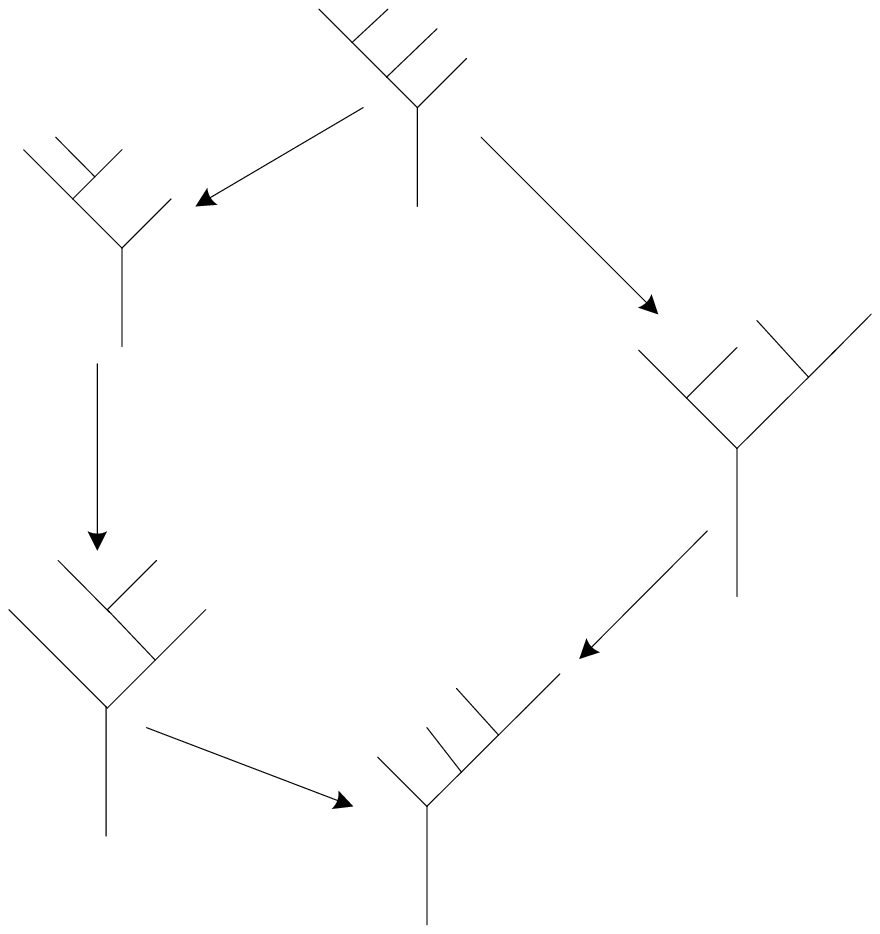}

Figure~\label{fig1}\ref{fig1}
\end{figure}
More precisely, for the critical monomial $\mu_i\circ_1\mu_j\circ_1\mu_k$, the left side of the pentagon is
$$\mu_i \circ_1 (\mu_j\circ_1\mu_k)
=\mu_i \circ_1 (\mu_j\circ_2\mu_k)
=(\mu_i \circ_1\mu_j)\circ_2 \mu_k
=(\mu_i \circ_2 \mu_j)\circ_2\mu_k
=\mu_i \circ_2(\mu_j\circ_1 \mu_k)
=\mu_i \circ_2 (\mu_j\circ_2\mu_k),
$$
and the right side of the pentagon is
$$(\mu_i\circ_1 \mu_j)\circ_1 \mu_k
= (\mu_i\circ_2 \mu_j)\circ_1\mu_k
=\mu_i\circ_2 (\mu_j\circ_2\mu_k).$$
Hence the critical monomial is confluent, as needed.
\end{proof}

\subsection{Free matching dialgebras}
\mlabel{sec:free}

The concept of a free matching dialgebra is defined by the usual universal property coming from the forgetful functor from the category of matching dialgebras to the category of sets. We first give in Section~\mref{ss:frees} a construction of the free matching dialgebra on a vector space by the rewriting method in Section~\mref{ss:matkos} which also gives the graded vector space of the operad of matching dialgebras. We then give in Section~\mref{ss:freea} constructions of the free matching dialgebras from tensor algebras.

\subsubsection{Free matching matching dialgebras}
\mlabel{ss:frees}

\begin{theorem}
Let $V$ be a vector space with a basis $X$ and let $\{\op_1,\op_2\}$ be a set. Define the disjoint union
\begin{eqnarray*}
\mds(X):&=& X\coprod \left(\coprod_{k\geq 1} X\times (\{\op_1,\op_2\}
 \times X)^k\right) \notag\\
 &=& \left\{\, u:=x_1\op_{i_1} x_2\op_{i_2} \cdots \op_{i_{k-1}} x_k\ |\ x_i\in X, 1\leq i\leq k,
 i_j\in \{1,2\}, 1\leq j\leq k-1, k\geq 1\right\}
\mlabel{eq:fmds}
 \end{eqnarray*}
 with the convention that $u=x_1$ when $k=1$.
Let
$$\mda(V):= \bfk \mds(X)$$
be the $\bfk$-linear span.
For $u=x_1\op_{i_1} \cdots \op_{i_{k-1}} x_k$
 and $v=x_{k+1}\op_{i_{k+1}} \cdots \op_{i_{k+\ell-1}} x_{k+\ell}$ in $\mds(X)$, define
 \begin{eqnarray}
 \mu_1(u, v):&=& x_1 \op_{i_1}\cdots\op_{i_{k-1}} x_k \op_{1} x_{k+1}\op_{i_{k+1}} \cdots \op_{i_{k+\ell-1}} x_{k+\ell}, \mlabel{eq:deop1}\\
 \mu_2(u,v):&=& x_1 \op_{i_1}\cdots\op_{i_{k-1}} x_k \op_{2} x_{k+1}\op_{i_{k+1}} \cdots \op_{i_{k+\ell-1}} x_{k+\ell}. \mlabel{eq:deop2}
 \end{eqnarray}
Extend $\mu_1$ and $\mu_2$ to $\mda(V)$ by bilinearity.
Then the triple $(\mda(V),\mu_1,\mu_2)$, together with the natural embedding
 $$j_V:V\to \mda(V), \quad v\mapsto v, \quad v\in V,$$
is the free matching dialgebra on $V$. More precisely, the following universal property is satisfied: Let $(R,\mu_1',\mu_2')$ be a matching dialgebra with
operations $\mu_1'$ and $\mu_2'$. Let $f:V\longrightarrow R$ be a linear map. Then there exists a unique homomorphism
$$\free{f}:\mda(V)\rightarrow R$$ of matching dialgebras such that
$f=\free{f} \circ j_V$.
 \mlabel{thm:fmds}
\end{theorem}

\begin{proof}
The rewriting system of the ns operad $As^{(2)}$ is shown to be convergent in the proof of Theorem~\mref{thm:kos}. Thus, as in \cite[Section 8.1]{LV}, the operad admits a PBW basis consisting of the right combs shown in Figure~\ref{fig} in page~\pageref{fig} with vertices decorated by $\mu_1$ and $\mu_2$.
\begin{figure}[h] \centering

  \includegraphics*[223,393][334,518]{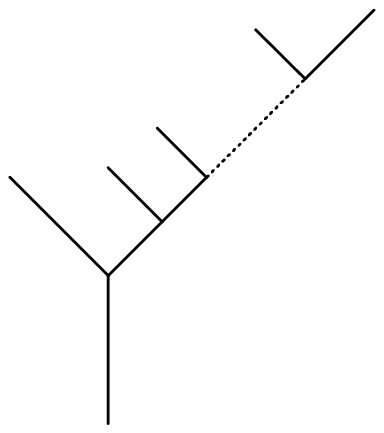}

Figure~\label{fig}\ref{fig}
\end{figure}
Thus $As^{(2)}_n$ has a $\bfk$-basis $\{(\mu_{i_1},\cdots,\mu_{i_{n-1}})\,|\, i_j\in \{1,2\}, 1\leq j\leq n-1\}$.
Since for a vector space $V$, the free $As^{(2)}$-algebra on $V$ is  $As^{(2)}(V)=\bigoplus_{n\geq 1} As^{(2)}_n \ot V^{\ot n}$, we see that the free matching dialgebra on $V$ has a basis given in the theorem.
The given multiplications can then be derived from the compositions of the operations of the operad $As^{(2)}$.
\end{proof}

\subsubsection{Free matching dialgebras in terms of tensor algebras}
\mlabel{ss:freea}
We now give two more constructions of free matching dialgebras from tensor algebras.

For a vector space $W$, let $\overline{T}_{\rop}(W)=\bigoplus_{k\geq 1}
W^{\ot k}$ denote the non-unitary tensor algebra where the (tensor) product is denoted by $\ot_\rop$.

\begin{theorem}
Let $V$ be a vector space over $\bfk$.
\begin{enumerate}
\item
Equip $\overline{T}_{\rop_1}(\overline{T}_{\rop_2}(V))$ with products $\mu_1$ and $\mu_2$
as follows. First define $\mu_1$ to be $\ot_{\rop_1}$. Next for
$u=u_1\ot_{\rop_1} \cdots \ot_{\rop_1}u_m$
and $v=v_1\ot_{\rop_1} \cdots \ot_{\rop_1}v_n$
 in
$\overline{T}_{\rop_1}(\overline{T}_{\rop_2}( V))$ with $u_i, v_j\in \overline{T}_{\rop_2}(V)$ for
$1\leq i\leq m$, $1\leq j\leq n,$ define
\begin{equation}
\mu_2(u,v):=u_1\ot_{\rop_1} \cdots \ot_{\rop_1}(u_m\ot_{\rop_2}v_1)\ot_{\rop_1} \cdots
\ot_{\rop_1} v_n.
\mlabel{eq:prod2a}
\end{equation}
Then $(\overline{T}_{\rop_1}(\overline{T}_{\rop_2}(V)),\mu_1,\mu_2)$ is the free matching dialgebra on
$V$. \mlabel{it:tma}
\item
Similarly, $\overline{T}_{\rop_1}(\overline{T}_{\rop_2}(V))$ with  products $\mu_2$ and
$\mu_2$ defined as follows is the free matching dialgebra on $V$. First, define
$\mu_2$ to be $\ot_{\rop_2}$. Next for $u=u_1\ot_{\rop_2} \cdots \ot_{\rop_2}u_m$ and
$v=v_1\ot_{\rop_2} \cdots \ot_{\rop_2}v_n$ in $\overline{T}_{\rop_1}(\overline{T}_{\rop_2}(V))$ with
$u_i, v_j\in \overline{T}_{\rop_1}(V)$ for $1\leq i\leq m$, $1\leq j\leq n,$
define
\begin{equation}
\mu_1(u,v):=u_1\ot_{\rop_2} \cdots \ot_{\rop_2}(u_m\ot_{\rop_1}v_1)\ot_{\rop_2} \cdots
\ot_{\rop_2}v_n.
\mlabel{eq:prod2b}
\end{equation}
\mlabel{it:tmb}
\end{enumerate}
\mlabel{thm:fmds2}
\end{theorem}
\begin{proof}
(\mref{it:tma}). We only need to show that $ \overline{T}_{\rop_1}(\overline{T}_{\rop_2}(V))\cong
\mda(V)$ as matching dialgebras, where $\mda(V)$ is the free
matching dialgebra defined in Theorem~\mref{thm:fmds}.

Construct homomorphisms $\varphi:\overline{T}_{\rop_1}(\overline{T}_{\rop_2}(V))\to
\mda(V)$ and $\psi:\mda(V)\to \overline{T}_{\rop_1}(\overline{T}_{\rop_2}(V))$ as follows.
Let \begin{equation} u=u_1\ot_{\rop_1} \cdots \ot_{\rop_1} u_m\in
\overline{T}_{\rop_1}(\overline{T}_{\rop_2}(V)) \,\,\text{with} \,\,u_i=u_{i1}\ot_{\rop_2}\cdots
\ot_{\rop_2} u_{i n_i}\in \overline{T}_{\rop_2}(V), 1\leq i\leq
m.\mlabel{eq:freemda1}
\end{equation}
 Define
\begin{eqnarray*}
\varphi(u): &=& (u_{11}\op_2
\cdots \op_2 u_{1n_1})\op_1\cdots \op_1(u_{m 1}\op_2 \cdots
\op_2 u_{mn_m}).
\end{eqnarray*}

Conversely, suppose $v=v_1\op_{i_1} \cdots \op_{i_{k-1}}v_k$ is in
$\mda(V)$ with $i_j\in \{1,2\}, 1\leq j\leq k-1$,
then there are $1=p_1<\cdots <p_\ell\leq k$ and $k_j\geq 0, 1\leq
j\leq \ell$ such that:
\begin{equation}
v=\overline{v}_1 \op_1 \cdots \op_1 \overline{v}_\ell,\mlabel{eq:freemda2}
\end{equation}
where
$\overline{v}_j=v_{p_j}\op_2 \cdots \op_2 v_{p_j+k_j}$.
So in particular $p_{j+1}=p_j+k_j+1$, $1\leq j\leq \ell-1$.
Then define
$$\psi(v):=\overline{u}_1\ot_{\rop_1} \cdots \ot_{\rop_1} \overline{u}_\ell,$$
where
$\overline{u}_j=x_{p_j}\ot_{\rop_2} \cdots \ot_{\rop_2} x_{p_j+k_j},\,1\leq
j\leq \ell.$

We next check that $\varphi$ and $\psi$ are inverse of each other.
For $u$ in $\overline{T}_{\rop_1}(\overline{T}_{\rop_2}(V))$ defined in
Eq.~(\mref{eq:freemda1}) and $v$ in $\mda(V)$ defined in Eq.~(\mref{eq:freemda2}), we have

\begin{eqnarray*}
(\psi\ope \varphi)(u) &=& \psi(u_{11}\ot_{\fopb} \cdots \ot_{\fopb}
u_{1n_1}\ot_{\fopa} \cdots
\ot_{\fopa} u_{m1}\ot_{\fopb} \cdots \ot_{\fopb} u_{mn_m})\\
&=& (u_{11} \op_2\cdots \op_2 u_{1n_1})\op_1 \cdots \op_1 (u_{m1}\op_2\cdots \op_2 u_{mn_m})\\
&=& u.
\end{eqnarray*}
Similarly
$\varphi\ope \psi=\rm id_{\mda(V)}$.
\end{proof}

\subsection{Operadic homology of matching dialgebras}
\mlabel{ss:mathom}

In Theorem~\mref{thm:kos}, we proved that the operad $As^{(2)}$
of matching dialgebras
 is Koszul. On the other hand, we know~\mcite{LV} that, for a Koszul operad, the Quillen homology is equal to the homology of the bar construction
defined in~\mcite{LV} of the algebra over an operad. So the Quillen
homology of the matching dialgebra is the desired operadic homology and, in order to give the operadic
homology of matching dialgebra,
we only need to give the
construction of the Koszul dual cooperad ${{As^{(2)}}^{\scriptstyle \textrm{!`}}}$.

\begin{prop}
The operad $As^{(2)}$ is a Koszul self dual operad. More generally, for $k\geq 2$, let $As^{(k)}$ denote the binary quadratic ns operad whose space of binary operations has a basis $\mu_i, 1\leq i\leq k$ and whose space $R$ of relations is spanned by
$$ \mu_i\circ _1 \mu_j - \mu_i\circ_2\mu_j, \quad 1\leq i, j\leq k.$$
Here we use the notation $\circ_1:=((\cdot,\cdot),\cdot)$,  $\circ_2:=(\cdot,(\cdot,\cdot))$
before Theorem~\mref{thm:kos}.
Then $As^{(k)}$ is a Koszul self dual operad.  \mlabel{pp:self}
\end{prop}

\begin{proof}
In the Koszul dual operad ${As^{(k)}}^{!}$ of $As^{(k)}$, let $\mu_i^*, 1\leq i\leq k,$ be the dual basis of $\mu_i, 1\leq i\leq k$. From \mcite{GK} and \cite[Chapter 7]{LV}, we find that
the space of relations of ${As^{(k)}}^{!}$, being the annihilator of $R$, is precisely the subspace with a basis
$$\left
\{\mu^*_i\circ_1\mu^*_j-\mu^*_i\circ_2\mu^*_j\,\big|\, 1\leq i,j\leq k
\right\}.$$
This proves the lemma.
\end{proof}

By the above fact and Theorem~\mref{thm:kos}, we can compute the Koszul complex for matching dialgebras.

\begin{coro}
Let $A$ be a matching dialgebra, i.e. a $As^{(2)}$-algenra.
The Koszul complex of $A$ is given by
$$\calc^{As^{(2)}}_{\op}(A):=({As^{(2)}}^{\scriptstyle \textrm{!`}}
(A),d).$$ More precisely, $\calc^{As^{(2)}}_{n-1}(A)={As^{(2)}}^{\scriptstyle \textrm{!`}}_{n}\ot A^{\ot n}$ with $d$
satisfying
\begin{eqnarray*}
d(\lambda\ot(a_1,\cdots,a_n))
&:=&\sum_{i=1}^{n-1}(-1)^{i+1}d_{i}(\lambda\ot(a_1,\cdots,a_n))\\
&:=&\sum_{i=1}^{n-1}(-1)^{i+1}
\xi\ot(a_1,\cdots,a_{i-1},\mu(a_i,a_{i+1}),\cdots,a_n)
\end{eqnarray*}
for $$\Delta_{(1)}(\lambda)=\sum_{i=1}^{n-1}(\xi;id,\cdots,id,\mu,id,\cdots,id)$$
 where $\mu\in E={\bf k}\{{\mu_1,\mu_2}\}$, $\lambda\in {As^{(2)}}^{\scriptstyle \textrm{!`}}_n$,
 $\xi\in {As^{(2)}}^{\scriptstyle \textrm{!`}}_{n-1}$
and $a_j \in A, 1\leq j\leq n$ and $\Delta_{(1)}$ is the infinitesimal
decomposition map in cooperad ${As^{(2)}}^{\scriptstyle \textrm{!`}}$. \mlabel{cor:mda}
\end{coro}

\begin{proof} Koszul complexes of symmetric operads are considered in~\cite[Chapter 12.1.2]{LV}. Since the operad $As^{(2)}$ is an ns operad, we give a direct proof.

Recall from Section~\cite[Chapter 11.2.2]{LV} that the
differential map $d$ is the unique coderivation which extends the composite of
the morphism $\kappa$ defined in ~\cite[Chapter 7.4.1]{LV} with the product
$\gamma_A$ given in ~\cite[Chapter 5.2.3]{LV} of A. It is explicitly given by the following composite
$${As^{(2)}}^{\scriptstyle \textrm{!`}}(A)\xrightarrow{\Delta_{(1)}\circ \text{Id}_{A}}({As^{(2)}}^{\scriptstyle \textrm{!`}}\circ_{(1)}{As^{(2)}}^{\scriptstyle \textrm{!`}})\circ \text{A}\xrightarrow{(\text{Id}{\circ_{(1)}}\kappa)\circ\text{Id}_{A}}
({As^{(2)}}^{\scriptstyle \textrm{!`}}\circ_{(1)}As^{(2)})\circ \text{A}\rightarrowtail {As^{(2)}}^{\scriptstyle \textrm{!`}}\circ As^{(2)}\circ \text{A}
\xrightarrow{\text{Id}\circ \gamma_A}{As^{(2)}}^{\scriptstyle \textrm{!`}}(A).$$

Since the operad $As^{(2)}$ is binary, the infinitesimal decomposition map $\Delta_{(1)}$
splits the $n$-ary cooperation $\lambda$ of ${As^{(2)}}^{\scriptstyle \textrm{!`}}$ into an $(n-1)$-ary cooperation and a binary cooperation
of ${As^{(2)}}^{\scriptstyle \textrm{!`}}$. By the morphism
$\kappa:{As^{(2)}}^{\scriptstyle \textrm{!`}}\rightarrow As^{(2)}$, the latter
cooperation is viewed as a binary operation $\mu$
of the operad $As^{(2)}$. The sum is over all these
possibilities of splitting.

Thus it is enough to check $d^{2}=0$. Since $\lambda$
is fixed, we just replace $d(\lambda\ot(a_1,\cdots,a_n))$
 by $d(a_1,\cdots,a_n)$ to simplify the notation.
By the definition of $d$ given in Corollary ~\mref{cor:mda}, we have
\begin{eqnarray}
d^{2}(a_1,\cdots,a_{n})
&=&\sum_{i=1}^{n-1}(-1)^{i+1}
\left(\sum_{j=1}^{n-2}(-1)^{j+1}(a_1,\cdots,a_{j-1},\nu(a_j,a_{j+1}),\cdots,a_{i-1},\mu(a_i,a_{i+1}),\cdots,a_{n})\right)\notag \\
&=&\sum_{i=1}^{n-1}\sum_{j=1}^{n-1}c_{ij}
\mlabel{eq:diff}
\end{eqnarray}
with $\nu,\mu\in\{\opa,\opb\}$ and
\[c_{ij}=\left \{
\begin{array}{ll}
0 & i=j,\\
(-1)^{i+j+1}(a_1,\cdots,a_{i-1},\nu(\mu(a_i,a_{i+1}),a_{i+2}),\cdots,a_n) & j=i+1,\\
(-1)^{i+j+1}(a_1,\cdots,a_{i-1},\mu(a_i,a_{i+1}),\cdots,a_{j-1},\nu(a_j,a_{j+1}),\cdots,a_n) & j\geq i+2,\\
(-1)^{i+j}(a_1,\cdots,a_{i-2},\nu(a_{i-1},\mu(a_i,a_{i+1})),\cdots,a_n) & i=j+1,\\
(-1)^{i+j}(a_1,\cdots,a_{j-1},\nu(a_j,a_{j+1}),\cdots,a_{i-1},\mu(a_i,a_{i+1}),\cdots,a_n) & i\geq j+2.
\end{array}
\right.
\]

If $j=i+1$, we have
$$c_{ij}=(-1)^{i+j+1}(a_1,\cdots,a_{i-1},\nu(\mu(a_i,a_{i+1}),a_{i+2}),\cdots,a_n)$$ and
\begin{eqnarray*}
c_{ji}&=& (-1)^{i+j}(a_1,\cdots,a_{j-2},\nu(a_{j-1},\mu(a_j,a_{j+1})),\cdots,a_n)\\
&=& (-1)^{i+j}(a_1,\cdots,a_{i-1},\nu(a_i,\mu(a_{i+1},a_{i+2})),\cdots,a_n)\\
&=&- (-1)^{i+j+1}(a_1,\cdots,a_{i-1},\nu(\mu(a_i,a_{i+1}),a_{i+2})\\
&=&-c_{ij}.
\end{eqnarray*}
If $j\geq i+2$, we also have
\begin{eqnarray*}
c_{ji}=(-1)^{i+j}(a_1,\cdots,a_{i-1},\nu(a_i,a_{i+1}),\cdots,a_{j-1},\mu(a_j,a_{j+1}),\cdots,a_n)=-c_{ij}.
\end{eqnarray*}
Thus we have $c_{ji}=-c_{ij}$ for all $1\leq i,j\leq n-1$ and hence $d^2=0.$
\end{proof}

\begin{coro}
Let $\bf k$ be a field of characteristic zero. The operadic homology of the free $As^{(2)}$-algebra $\bar{T}_{\rop_2}(\bar{T} _{\rop_1}(V))$ given in Theorem~\mref{thm:fmds2} is equal to
 \[
 {\mathrm{H}_{n}}^{As^{(2)}}(\overline{T}_{\rop_2}(\overline{T} _{\rop_1}(V))):=\left\{
   \begin{array}{ll}
V, & \text{if } \,n=0,\\
0, & \text{if } \,n\geq 1.
\end{array}
   \right.
\]
 \end{coro}

\begin{proof}
By Theorem~\mref{thm:kos}, $As^{(2)}$ is Koszul. Then the corollary is a direct consequence of \cite[Theorem~12.1.6]{LV} since $\overline{T}_{\rop_2}(\overline{T} _{\rop_1}(V))$
is the free $As^{(2)}$-algebra on $V$ by Theorem~\mref{thm:fmds2}.
\end{proof}

\smallskip

\noindent
{\bf Acknowledgements. }
C. Bai would like to thank the support by NSFC (10920161) and SRFDP
(200800550015). L. Guo thanks NSF grant DMS-1001855 for support. The authors thank J.-L. Loday, to whose memory this paper is dedicated to, and the editor for very helpful suggestions and comments.

\end{document}